\documentclass{elsarticle}

\usepackage[utf8]{inputenc}
\usepackage[english]{babel}

\usepackage{amsmath}
\usepackage{amssymb}
\usepackage{amsthm}
\usepackage{mathabx}
\usepackage{amsthm}
\usepackage{enumitem}
\usepackage{xcolor}
\usepackage{cancel}

\usepackage{orcidlink} %
\usepackage{tikz,xcolor,hyperref}

\makeatletter
\def\ps@pprintTitle{%
 \let\@oddhead\@empty
 \let\@evenhead\@empty
 \let\@oddfoot\@empty
 \let\@evenfoot\@empty
}
\makeatother

\newtheoremstyle{myaxiom}
  {0.3\baselineskip plus 0.2ex minus 0.1ex} 
  {0.3\baselineskip plus 0.2ex minus 0.1ex} 
  {}        
  {2em}        
  {\bfseries} 
  {:}        
  {1em}     
  {\thmname{#1}~\thmnumber{#2}\thmnote{(#3)}}

\theoremstyle{myaxiom}
\newtheorem{axiom}{Axiom}[section]

\newtheorem{Definition}{Definition}[section]
\newtheorem{theorem}{Theorem}[section]
\newtheorem{lemma}{Lemma}[section]
\usepackage[autostyle]{csquotes}
\usepackage{comment}

\begin{document}

\title{On the consistency of $\mathit{NF}$ via \emph{Fuzzy Forcing}}


\author[1]{Nicolás \textsc{Sevilla Simón}\corref{cor1}\orcidlink{0009-0004-4887-0483}}
\ead{nicolsev@ucm.es}
\address[1]{Faculty of Philosophy, Universidad Complutense, Madrid, Spain}
\cortext[cor1]{Corresponding author}

\begin{abstract}
In this paper, we present a proof of the consistency of the \emph{New Foundations} set theory ($\mathit{NF}$). $\mathit{NF}$'s main idea is to permit very large sets (including the Universal Set) by restricting set formation to stratified formulas, thereby avoiding the classic set-theoretic paradoxes. Our proof employs  a new forcing method incorporating concepts from fuzzy logic.
A brief outline of the proof can be as follows: (1) We extend $ZF$ to \emph{Fuzzy} $\mathit{ZF}$ with a membership function $\mu$ over $D=\mathbb{Q} \cap [0,1]$; (2) we define \emph{Fuzzy} $\mathit{NF}$ as $\Sigma$, and (3) we derive a crisp $\mathrm{N}$ model of $NF$. Our proof does not depend on Holmes' Tangled Type Theory ($\mathit{TTT}$). It establishes that if $\mathit{ZF}$ is consistent, then $\mathit{NF}$ is also consistent. It achieves that via the chain $\mathit{ZF} \rightarrow$ \emph{Fuzzy} $\mathit{ZF} \rightarrow \Sigma \rightarrow \mathit{NF}$. The method presented in this paper offers a novel perspective connecting fuzzy logic with classic set theory.
\end{abstract}


\date{}

\maketitle

\section{Introduction}
\label{Introduction}
In 1937, Quine published in New Foundations for Mathematical Logic the \emph{New Foundations set theory}  ($\mathit{NF}$) ~\cite{GlossarWiki:Quine:1937} , where the author introduced a finitely axiomatizable and poorly founded set theory. It was initially proposed as an attempt to simplify Bertrand Russell and Alfred N. Whitehead's type theory developed in Principia Mathematica~\cite{Russell1910-RUSPMV}, and was also introduced as an alternative to Zermelo-Fraenkel's theory ($\mathit{ZF}$)~\cite{GlossarWiki:Zermelo:1908b}. The most interesting thing to date may be the impossibility of defining a Universal Set in $\mathit{ZF}$, in contrast to its existence in $\mathit{NF}$.

$\mathit{NF}$ defines a Universal Set without falling into contradiction through a comprehension scheme restricted to stratified formulas (instead of the Axiom of Separation, which is used in $\mathit{ZFC}$). Thus, paradoxes such as Russell's are avoided in $\mathit{NF}$. This is achieved through a syntactic restriction rather than a cumulative hierarchy. This flexibility is of interest both because of its properties and the theorems that may be provable from it. Furthermore, its foundational role and its connection to other mathematical disciplines are also of particular interest.



But for all of this to be possible, we must first address an essential question that has remained open for nearly 90 years: \emph{is $\mathit{NF}$ consistent?}


Previous advances in the consistency of $\mathit{NF}$ have been significant but limited. Specker~\cite{Specker1954-SPETAO-14} showed that $\mathit{NF}$ refutes the Axiom of Choice and is consistent with the existence of non-standard models. Jensen~\cite{9100459b-4987-3c14-a695-06573776c8fc} proved the consistency of a restricted variant of $\mathit{NF}$, denoted $\mathit{NFU}$. Note that $\mathit{NFU}$ uses \emph{urelements}.
For decades, the consistency of $\mathit{NF}$ remained unsolved, although Holmes~\cite{holmes1998elementary} conjectured its feasibility and developed proofs based on Tangled Type Theory ($\mathit{TTT}$). In 2024, Holmes and Wilshaw~\cite{holmes2024nfconsistent} appear to have established the consistency of $\mathit{NF}$ via a Lean formalization~\cite{10.1007/978-3-030-79876-5_37}, a significant advance awaiting final peer review.

\subsection{Organization of the paper}
The paper is organized as follows: Section~\ref{Introduction} presents a brief historical introduction to the problem of the consistency of $\mathit{NF}$; 
Section~\ref{Definitions} defines \emph{Fuzzy} $\mathit{ZF}$  and \emph{Fuzzy} $\mathit{NF}$; Section~\ref{modfinitos} proves finite models for $\Sigma$; Section~\ref{modinfinitos} constructs $\mathcal{M}$ via compactness; Section~\ref{extracModNitido} extracts $\mathrm{N}$ and verifies $\mathit{NF}$; and finally, Section~\ref{conclusions} discusses implications and future work.

\section{Definition of \emph{Fuzzy} $\mathit{ZF}$ and \emph{Fuzzy} $\mathit{NF}$}
\label{Definitions}
Throughout this paper, we use the following notation in the context of fuzzy set theory. $X$ is a reference domain over which fuzzy sets are constructed. $A_i$ denotes a particular fuzzy set. $V_n$ denotes the $n$-th level of the cumulative hierarchy of sets in $\mathit{ZF}$. $\Phi$ is a first-order formula for defining fuzzy sets. $\Sigma$ is the complete theory of \emph{Fuzzy} $\mathit{NF}$. Finally, $\Sigma_n$ denotes a finite subset of $\Sigma$ indexed by $n$. $U_X$ is defined as $U_X = \{\mu:X \rightarrow D\}$, that is, as the set of all fuzzy membership functions defined over $X$.  In summary, it constitutes the universe containing the fuzzy sets. $V$ is the universe of crisp sets. $V_X$ defines the global fuzzy universe, consisting of all subsets of the entire classical universe $V$. Formally, when $X=V, \: V_X=U_V=\{\mu:V \rightarrow D\}$.


The fuzzy logic used in this framework is a first-order classical logic that extends the system with a primitive symbol denoted as $\mu(x, A) = d$. This symbol represents the degree of membership, on a rational scale, of an element $x$ to a fuzzy set $A$. Semantically, $\mu$ is defined as a function, which assigns to each pair $(x, A)$ a rational value $d$, effectively encoding the fuzzy membership within the model.

Syntactically, $\mu$ functions as a ternary relational predicate $(x, A, d)$, allowing for the quantification of membership degrees directly within the logical language. This differs from Zadeh's framework, where $\mu$ is treated merely as a function. In this formulation, $\mu$ is regarded as a primitive logical symbol. Consequently, this approach facilitates the integration of fuzzy sets into the logical syntax of the system.

In contrast to algebraic fuzzy logics, we don't define t-norms in this context. Instead, the connectives are interpreted using classical semantics, meaning they function as Boolean functions. As a result, the models of this new [pseudo] fuzzy logic are still classical and first-order, but they are enhanced with fuzzy membership functions that represent fuzzy sets. Although the truth values of the formulas remain Boolean, the membership of elements in sets is not Boolean. This is because the fuzzy aspect pertains to the content of the sets, not the logical framework used to describe them. In addition, the deductive process remains classical.

\emph{Fuzzy} $\mathit{ZF}$ is, in essence, $\mathit{ZF}$ extended with our fuzzy logic framework. This constitutes a (seemingly minor yet significant) innovation relative to Zadeh’s fuzzy set theory: whereas Zadeh grounded his formalism in function theory to define fuzzy sets external to first-order logic, we formalize fuzzy membership \emph{within} the logical system itself, thereby directly integrating fuzzy sets into the logical syntax. This integration enables internal reasoning about the membership degrees that characterize fuzzy sets.

\emph{Fuzzy} $\mathit{NF}$. \emph{Mutatis mutandis}, it retains the core of $\mathit{NF}$ but allows fuzzy sets defined via stratified formulas to have fuzzy content. In \emph{Fuzzy} $\mathit{NF}$, all elements are represented as fuzzy sets.


 

\subsection{\emph{Fuzzy} $\mathit{ZF}$}

\begin{Definition}
We define \emph{Fuzzy} $\mathit{ZF}$ as an extension of $\mathit{ZF}$ that incorporates \emph{fuzzy} sets over crisp domains.
\end{Definition}
The formal theory of \emph{Fuzzy} $\mathit{ZF}$ consists of a language and a set of axioms based on first-order logic. Let $\mathcal{L}_1$ represent its formal language:

\begin{itemize}
\item $\in$ : crisp membership.

\item A domain $D=\mathbb{Q} \cap [0,1]$ is countable and equipped with the usual order $<$ and equality $=$ relations.

\item $\mu (x, A)$ : Fuzzy membership function, which assigns to each pair (x,A) a value of $D$, where $x$ is an element of a crisp domain and $A$ is a fuzzy set.


\item Standard logical symbols in first-order logic, connectives ($\wedge$, $\vee$, $\neg$, $\rightarrow$) and quantifiers ($\forall$, $\exists$).

\item $\mathit{U}$-sort: Domain of set elements (variables $\mathit{x}, \mathit{y}, \mathit{z}, \ldots$).

\item $\mathit{D}$-sort: Domain of membership degrees (variables $\mathit{v}, \mathit{d}, \ldots$).
\end{itemize}

This language provides the foundation for expressing the axioms and theorems of \emph{Fuzzy} $\mathit{ZF}$.

And let its axioms be the following:

\hspace*{0.3cm}\textbf{Axioms of $\mathit{ZF}$}:
Extensionality, Separation, Union, Power, Infinity, Replacement, Foundation (for crisp sets). From these axioms we provide a consistent basis for constructing a crisp set. 

\begin{axiom}{Existence of \emph{Fuzzy} Sets.}
\label{ax:existConjDif}

For every crisp set \( X \), there exists a set \( U_{X} = \{\mu: X \rightarrow D\} \). This set \( U_{X} \) includes all possible fuzzy membership functions defined on \( X \). With this axiom, we introduce the concept of fuzzy levels into \emph{Fuzzy} $\mathit{ZF}$, enabling fuzzy sets to exist as formal objects that we can manipulate.
The Power Set and Replacement Axioms allow one to construct $U_{X}$ in $\mathit{ZF}$.
\end{axiom}

\begin{axiom}{\emph{Fuzzy} Extensionality.}
\label{ax:DiffExtensionality}
$$
\forall X \, \forall A \in U_X \, \forall B \in U_X (\forall x \in X (\mu (x,A)=\mu (x,B)) \rightarrow A=B)
$$
This axiom specifies that two fuzzy sets $A$ and $B$ over the same domain $X$ are equal \emph{iff} they have the same membership function on all objects in $X$.

It is a natural extension of the $\mathit{ZF}$ Extensionality Axiom, but applied to $\mu$ functions instead of the relation $\in$. It is consistent with the fuzzy set theory of Zadeh et al. but adapted to our context.
\end{axiom}

\begin{axiom}{\emph{Fuzzy} Comprehension.}
\label{AxCompDifZF}

There exists $A \in U_{X}$ such that for every formula $\Phi (x, v)$ in the language $\mathcal{L}_1$, with $x$ as a set variable and $v$ as a degree variable from the domain $D$ of membership degrees, and for every set $X$, there exists a fuzzy set $A$ such that the membership degree of $x \in X$ in $A$ is determined by $\Phi(x,v)$. Formally:
\begin{align*}
\exists A \, \forall x \in X \, \forall d \in D (\mu (x,A) = d \leftrightarrow (\Phi (x,d) \land \\
\forall v \in D (v > d \rightarrow \, \lnot \Phi (x, v)) \, \land \exists v \in D (v \leq d \land \Phi (x, v))) \lor \\ 
(d=0 \land \forall v \in D \lnot \Phi_i (x,v)))
\end{align*}

This axiom allows us to define fuzzy sets based on properties definable in formulas. It facilitates us to generate fuzzy sets in \emph{Fuzzy} $\mathit{ZF}$, which is essential to model the stratified comprehension instances of \emph{Restricted} \emph{Fuzzy} $\mathit{NF}$ \footnote{This term is associated with fuzzy sets and comes with specific restrictions. This is detailed in Section~\ref{modfinitos}.} in $V_{n}$.
We do not restrict $\Phi(x,v)$ to stratification (as we will later do in \emph{Fuzzy} $\mathit{NF}$), because we want to allow flexibility in \emph{Fuzzy} $\mathit{ZF}$.


The formulation of \emph{Fuzzy} $\mathit{ZF}$ does not reproduce Russell's paradox. Unlike classical bivalent logic, where unrestricted comprehension leads to contradictions such as $R=\{x \: \mid \: x \notin x\}$, fuzzy logic avoids these constructs through a membership semantics with graded values.

The membership value denoted as \(\mu(x,A) = 1\) can be informally understood as indicating full membership. However, it should not be directly identified with the classical membership relation $x \in A$ unless the model explicitly defines this equivalence~\cite{10.5555/202684}.

In the context of \emph{Fuzzy} $\mathit{ZF}$, the $\mathit{ZF}$ axioms do not automatically apply when $\mu(x,A)=1$. Specifically, neither the separation schema nor the Foundation Axiom is activated, as would be the case in a classical framework.

This semantic distinction is crucial, since not equating $\mu(x,A)=1$ with $x \in A$ prevents the direct transfer of $\mathit{ZF}$ restrictions to fuzzy sets. Consequently, the Fuzzy Comprehension Axiom can generally be formulated without generating self-referential paradoxes, since membership functions are constructed in terms of achievable maximum degrees~\cite{Baldwin1980-BALTRO-2}.

There is no self-referential comprehension at the construction level. Fuzzy sets are defined by functions on crisp domains, and the function $\mu$ is constructed by applying fuzzy comprehension to formulas that depend on $x$ and $d$, but not directly on the function $\mu$ itself. This prevents the emergence of hazardous self-referential definitions, since $\Phi(x, d)$ is not permitted to refer to $\mu(X, A)$. This axiom authorizes quantification only over variables $x$ in the crisp domain and values $d$ in $D$. Hence, when constructing $U_X$, we rely exclusively on local properties of the elements of $X$ and the values in $D$.
\end{axiom}

In summary, the axioms of $\mathit{ZF}$ and the Existence Axiom of $U_X$ (Axiom~\ref{ax:existConjDif}) ensure that \emph{Fuzzy} $\mathit{ZF}$ is a natural extension of $\mathit{ZF}$, inheriting its consistency. Crisp sets constitute the foundation, while fuzzy sets are defined as an extension of $\mathit{ZF}$, preserving its structure without altering it. The axioms Fuzzy Extensionality (Axiom ~\ref{ax:DiffExtensionality}) and Fuzzy Comprehension (Axiom ~\ref{AxCompDifZF}) allow us to define $U_{X}$ and generate fuzzy sets, such as $V_{X}$ (universal within the finite domain $X$ in the context of \emph{Restricted Fuzzy} $\mathit{NF}$) and $A_{i}$, which model the stratified formulas of \emph{Fuzzy} $\mathit{NF}$ in domains such as $V_{n}$. 
It is essential for building finite models, as will be explained later.
The choice of countable $D$ guarantees a countable language, a highly recommended requirement for the Compactness Theorem. Unlike the theories of Zadeh~\cite{ZADEH1965338}, Gödel~\cite{Goedel1931}, and \L{}ukasiewicz~\cite{Lukasiewicz1988-UKAOLT-2}, we do not include specific fuzzy operations (e.g., \emph{min}, \emph{max}). 

The conservativity of \emph{Fuzzy} $\mathit{ZF}$ is established through the following criterion: \emph{Fuzzy} $\mathit{ZF}$ is considered a conservative extension of $\mathit{ZF}$ if every classic formula $\Phi$ in the language of $\mathit{ZF}$ that can be proven using the axioms of \emph{Fuzzy} $\mathit{ZF}$ is also provable from the axioms of $\mathit{ZF}$ alone. This conservativity criterion is met because the new axioms of \emph{Fuzzy} $\mathit{ZF}$ (specifically, Axiom~\ref{ax:existConjDif}, Axiom~\ref{ax:DiffExtensionality}, and Axiom~\ref{AxCompDifZF}) only expand the framework by introducing fuzzy sets while preserving all the properties and provable theorems related to crisp sets.

Fuzzy sets remain ontologically distinct from the crisp ones and do not interact with the $\in$-relations governing them, thus maintaining the integrity of classic set-theoretic foundations.


\begin{theorem}
Relative consistency of \emph{Fuzzy} $\mathit{ZF}$.

If $\mathit{ZF}$ is consistent, then \emph{Fuzzy} $\mathit{ZF}$ is consistent.

\begin{proof}
We assume that $M=(V_n,\in_M)$ is a model of $\mathit{ZF}$. Our goal is to construct, within $M$, a structure $M^{*}$ that satisfies all axioms of \emph{Fuzzy} $\mathit{ZF}$.

\subsection*{1. Construction of the Universe}
Let $D$ be a set in $M$ (by definability). Define $V_X = \{\mu \mid \mu: V_n \rightarrow D\}$, the collection of all functions from $V_n$ to $D$. This is a set in $M$ by the Replacement and Power Set axioms. The universe of \emph{Fuzzy} $\mathit{ZF}$ is the disjoint union:
\[
U = V_n \sqcup V_X,
\]
where $V_n$ is the domain of crisp sets and $V_X$ is the domain of fuzzy sets.

\subsection*{2. Definition of Membership Relations}
\begin{itemize}
    \item \textbf{Classical membership:} Interpreted as $\in_M$ on $V_n$.
    \item \textbf{Fuzzy membership:} For $x \in V_n$ and $A \in V_X$, define $\mu(x, A) = A(x)$. We do not define $x \in A$ for $A \in V_X$, nor do we identify $\mu(x, A) = 1$ with $x \in A$. This ensures a strict semantic and ontological separation between the crisp and fuzzy universes.
\end{itemize}

\subsection*{3. Verification of Axioms}
We verify that $M^{*} = (U, \in, \mu, D)$ satisfies all axioms of \emph{Fuzzy} $\mathit{ZF}$.

\subsubsection*{A. Classical $\mathit{ZF}$ Axioms}
Restricted to $V_n$, these are satisfied because $(V_n, \in_M)$ is a model of $\mathit{ZF}$ by hypothesis.

\subsubsection*{B. Existence of Fuzzy Sets}
For any $X \in V_n$, the set $U_X = \{\mu: X \rightarrow D\}$ exists in $M$ by the Power Set and Replacement axioms. Since $X \subseteq V_n$ and $D \in V_n$, $U_X$ is a definable subclass of $V_X$ and hence a set in $V_n$.

\subsubsection*{C. Extensionality for Fuzzy Sets}
For all $\mu_1, \mu_2 \in U_X$, if $\mu_1(x) = \mu_2(x)$ for all $x \in X$, then $\mu_1 = \mu_2$ as functions, by the extensionality of functions in $\mathit{ZF}$.

\subsubsection*{D. Comprehension for Fuzzy Sets}
Let $\Phi(x,d)$ be a formula in the language of \emph{Fuzzy} $\mathit{ZF}$. Define a fuzzy set $A \in V_X$ by:
\[
\mu (x,A) = d \Leftrightarrow (\Phi (x,d) \land \forall v > d \: \lnot \Phi (x,v)) \lor (d = 0 \land \forall v \in D \: \lnot \Phi (x,v))
\]
This is well-defined in $M$ by Separation and Replacement. The biconditional required by Axiom \ref{AxCompDifZF} holds by construction.

\subsection*{Conclusion}
The structure $M^{*}$ constructed within $M$ satisfies all axioms of \emph{Fuzzy} $\mathit{ZF}$. Therefore, if $\mathit{ZF}$ is consistent, then so is \emph{Fuzzy} $\mathit{ZF}$.
\end{proof}
\end{theorem}

\subsection{\emph{Fuzzy} $\mathit{NF}$}

\begin{Definition}
\emph{Fuzzy} $\mathit{NF}$ adapts $\mathit{NF}$ to a fuzzy framework, with a unique universe $U$, which contains all fuzzy elements and sets. 
Unlike \emph{Fuzzy} $\mathit{ZF}$, there is no prior distinction between crisp and fuzzy domains; $U$ is heterogeneous and defined by the axioms themselves.
\end{Definition}

Its language $\mathcal{L}_2$ (in the framework of first-order logic) includes:

\begin{itemize}
\item $\mu (x,A)$ : binary fuzzy membership function, which assigns to each pair %
$(x,A) \in U$ a value in countable $D = \mathbb{Q} \cap [0,1]$, with relations $<$ and $=$ defined on domain $D$.
\item Standard logic symbols in first-order logic:
\begin{enumerate}
\item connectives: $\wedge$, $\vee$, $\neg$, $\rightarrow$
\item quantifiers: $\forall$, $\exists$
\item $\mathit{U}$-sort: Domain of set elements (variables $\mathit{x}, \mathit{y}, \mathit{z}, \ldots$).
\item $\mathit{D}$-sort: Domain of membership degrees (variables $\mathit{v}, \mathit{d}, \ldots$).necessary to quantify over them. 

\end{enumerate}
\end{itemize}

Its axioms, in first-order logic, are as follows:

\begin{axiom}{\emph{Fuzzy} Extensionality}.
\label{fuzzyExtensionality}
$$
\forall A \, \forall B \, \forall x \, (\mu (x,A) = \mu (x,B)) \rightarrow A=B
$$
Note here that universal quantification $\forall x$ reflects the unrestricted universe of \emph{Fuzzy} $\mathit{NF}$, unlike \emph{Fuzzy} $\mathit{ZF}$, where it is restricted to $X$.
\end{axiom}

\begin{axiom}{\emph{Stratified Fuzzy} Comprehension.}
\label{ax:StrFuzzyComp}

For each stratified formula $\Phi_{i}(x,v)$ where $v \in D$, and the stratification applies to $\in$:
\begin{align*}
 \exists A \, \forall x \, \forall d \, \in D \, (\mu (x,A) = d \leftrightarrow 
 (\Phi _{i} (x,d) \: \land \\ 
 \forall v \in D \, (v > d \rightarrow \, \lnot \Phi _{i} (x,v)) \, \land \exists v \in D (v \leq d \, \land \Phi_i (x,v))) \: \lor \\ 
  (d = 0 \: \land \forall v \in D \: \lnot \Phi_i (x,v))) 
\end{align*}

A formula $\Phi_{i}(x,v)$ is stratified \emph{iff}, each variable can be assigned a type such that: $x \in z$ implies $\emph{type(x)} = \emph{type(z)} - 1$.

With this axiom, we generalize the comprehension of $\mathit{NF}$ to the fuzzy case. Through this mechanism, we will generate all fuzzy sets in $U$, which will be essential, as we will see later, for the infinite model $\mathcal{M}$ to contain the required $A_{i}$.
Quantification of $x$, $d$, $v$ maintains the axiom in first-order logic, with a countable number of instances (one for each $\Phi_{i}$).

\end{axiom}

\begin{axiom}{\emph{Fuzzy} Universality}.

$$
\exists V \, (\forall x \in U \, (\mu (x,V) = 1) \land \mu (V,V) = 1)
$$

This axiom postulates the existence of a fuzzy set $V$ in $U$ that completely contains (with degree $\mu = 1)$ all elements of $U$, including itself. It will be essential for the crisp submodel $N$ to inherit this property for $\mathit{NF}$.
It is also a first-order sentence suitable for applying compactness to it.
\end{axiom}

The complete set of axioms for \emph{Fuzzy} $\mathit{NF}$, expressed in first-order logic, will be denoted as $\Sigma$, which is a countable set of axioms because there is one for extensionality, one for universality, and $\aleph_0$ for comprehension (a countable number $\aleph_0$ of comprehension instances, one for each stratified $\Phi_{i}$). Finally, $\Sigma$ represents the complete theory.

\section{Models of Finite Fragments of \emph{Restricted Fuzzy} $\mathit{NF}$ in \emph{Fuzzy} $\mathit{ZF}$}
\label{modfinitos}

In this section, we aim to demonstrate that every finite subset $\Sigma_{n} \subseteq \Sigma$ of the axioms of \emph{Fuzzy} $\mathit{NF}$ possesses a constructible model within \emph{Fuzzy} $\mathit{ZF}$. By Gödel's Completeness Theorem, this implies that such finite subsets are consistent, as any set of first-order formulas that has a model is consistent in the proof-theoretic sense. This result establishes the finite satisfiability of $\Sigma$, thereby enabling the application of the Compactness Theorem. As a consequence, we can conclude the existence of an infinite model of \emph{Fuzzy} $\mathit{NF}$, since every finite subset is satisfiable and thus the entire set $\Sigma$ is satisfiable as well.

\begin{Definition}
We define \emph{Restricted Fuzzy} $\mathit{NF}$ on a crisp set $X$ within \emph{Fuzzy} $\mathit{ZF}$ by means of the following axioms:
\end{Definition}

\begin{axiom}{\emph{Restricted Fuzzy} Extensionality.}
$$
\forall A \in U_X \, \forall B \in U_X (\forall x \in X \, (\mu(x,A) = \mu (x,B)) \rightarrow A=B 
$$
\end{axiom}
That is, if two fuzzy sets $A$ and $B$ in $U_X$ have the same degree of membership for every element of $X$, they are equal.

\begin{axiom} {\emph{Restricted Fuzzy Stratified} Comprehension.}
\label{ax:ResFuzzStraCom}

For each stratified formula $\Phi_{i}$ in a finite subset $\Sigma_{n} \subseteq \Sigma$, there exists $A \in U_X$ such that:
\begin{align*}
 \exists A \, \forall x \, \forall d \, \in D \, (\mu (x,A) = d \leftrightarrow 
 (\Phi _{i} (x,d) \: \land \\ 
 \forall v \in D \, (v > d \rightarrow \, \lnot \Phi _{i} (x,v)) \, \land \exists v \in D (v \leq d \, \land \Phi_i (x,v))) \: \lor \\ 
 (d = 0 \: \land \forall v \in D \: \lnot \Phi_i (x,v))) 
\end{align*}
\end{axiom}

\begin{axiom}{\emph{Restricted Fuzzy} Universality.} \label{ax:resFuzzyUniversality}
$$
\exists V_X \in U_X(\forall x \in X \:\mu (x,V_X) = 1 ) \land \mu (V_X, V_X) = 1)
$$
\end{axiom}

\begin{Definition}
Let $X=V_n$ be the $n$ level of the cumulative hierarchy in $\mathit{ZF}$, with $n$ finite and large enough to contain the elements necessary to satisfy $\Sigma_n$. In $\mathit{ZF}$, $V_n$ is a finite and well-founded set, with $\lvert V_n \rvert $ finite, which provides a suitable domain for $U_X =\{\mu : X \rightarrow D\}$ in \emph{Restricted Fuzzy} $\mathit{NF}$.
\end{Definition}

\begin{lemma}{Model Construction of $U_X$.}

By the Existence Axiom of \emph{fuzzy} sets (Axiom~\ref{ax:existConjDif}) in \emph{Fuzzy} $\mathit{ZF}$, we define $U_X = \{ \mu : V_n \rightarrow D \}$, where $X=V_n$. \begin{proof}
Since $\lvert V_n \rvert$ is finite (say $k$) and $\lvert D \rvert =\aleph_0$, the cardinality of $U_X$ is:
$$
\lvert U_X \rvert = \lvert D \rvert^{\lvert V_n \rvert}=(\aleph_0)^k = \aleph_0
$$
Note that, in transfinite cardinal arithmetic, $(\aleph_0)^k = \aleph_0$ for finite $k$. In \emph{Fuzzy} $\mathit{ZF}$, $U_X$ is a well-defined set constructed by the Replacement and Power Axioms since $D$ is countable. $U_X$ contains all possible fuzzy functions from $V_n$ to $D$, including those required for $V_X$ and $A_i$ needed by $\Sigma_n$. Although $U_X$ is infinite, the model ($V_n$, $U_X$) in $\Sigma_n$ uses a finite subset of $U_X$, as detailed in the following. \end{proof}
\end{lemma}

\begin{Definition} \label{DefModeloMn}
Consider $M_n$ the model that satisfies:  
$$
M_n =(U_X,V_n,\mu_M,D) \quad \text {for} \quad \Sigma_n, X=V_n
$$
where:
\begin{enumerate}
\item Domain and \emph{Fuzzy} Universe.
The crisp domain, $V_n$, is finite and well-founded in $\mathit{ZF}$. The fuzzy universe is $U_X =\{\mu:X \rightarrow D\}$, and has cardinality of $\aleph_0$. We consider a finite subset $S=\{V_X,A_1,...,A_k\} \subseteq U_X$, where $V_X$ satisfies the property of universality, and each $A_i$ corresponds to a $\Phi_{i}$ in $\Sigma_{n}$ with $k=\lvert \Sigma_{n} \rvert - 2$ (excluding extensionality and universality).


\item Function $\mu_{M}$.
Where $V_X$ is: $\mu_M (x, V_X) = 1$ for all $x \in V_n \cup S$ that is definable in $U_X$ as the constant function  for all $x$, that is, $\mu (x) = 1$.
For each $A_i$ corresponding to $\Phi_i \in \Sigma_n$:
\begin{align*}
\mu_M (x, A_i) = d \Leftrightarrow (\Phi_i (x,d) \, \land \forall v \in D (v > d \rightarrow \\
 \lnot \Phi_i (x, v)) \, \land \, \exists v \in D (v \leq d \, \land \Phi_i (x, v)))
\end{align*}


\end{enumerate}
The $M_n$ model uses $S$, a finite subset of $U_X$ of size $k+1$, ensuring that $\Sigma_{n}$ is satisfied.
\end{Definition}

We verify that $M_n = (U_X,V_n, \mu_{M} ,D)$ satisfies the axioms of \emph{Restricted Fuzzy} $\mathit{NF}$ for $\Sigma_{n}$:

\begin{lemma} {Verification of \emph{Restricted Fuzzy} Extensionality Axiom.}

In the model $M_n$ constructed over \emph{Fuzzy} $\mathit{ZF}$, it holds that for all $A, B \in U_X$:
$$
\forall x \in V_n, \mu_M (x, A) = \mu_M (x,B) \rightarrow A = B
$$
\begin{proof}


Let $V_n = \{x_1, \ldots, x_k\}$ be a finite set of variables and $D$   the set of membership values.

By definition:
$$
U_X = \{\mu : V_n \rightarrow D\}
$$

is the set of all possible membership functions over the finite domain $V_n$.

Let $M_n$ be a model based on \emph{Fuzzy} $\mathit{ZF}$, where the membership relation is given by:
$$
\mu_M : V_n \times U_X \rightarrow D
$$
and represents the membership value of the element $x \in V_n$ in the set $A \in U_X$.

Let $A$ and $B$ be elements of $U_X$. By definition, both $A$ and $B$ are functions from $V_n$ to $D$. In particular:
\begin{itemize}
\item $\mu_M(x, A)$ denotes the value assigned to $x$ by the function $A$, that is, $\mu_M(x, A) = A(x)$.
\item $\mu_M(x, B)$ denotes the value assigned to $x$ by the function $B$, that is, $\mu_M(x, B) = B(x)$.
\end{itemize}

Then the hypothesis:
$$
\forall x \in V_n, \mu_M(x,A) = \mu_M (x,B)
$$
is simply translated as:
$$
\forall x \in V_n, \: A(x) = B(x)
$$
Given that $V_n$ is finite, the equality of functions can be established by verifying pointwise equality at every element of $V_n$. Thus, $A = B$ as functions; that is, they are equal as elements of $U_X$.

Formally, in set theory, two functions $f, g : X \rightarrow Y$ are equal if and only if:
$$
\forall x \in X, f(x) = g(x)
$$
This principle of extensionality for functions is justified in \emph{Fuzzy} $\mathit{ZF}$ by Axiom~\ref{ax:DiffExtensionality}.

Since $\mu_M (x, A) = \mu_M (x, B)$ for all $x \in V_n$, it follows by the axiom that $A = B$ in $U_X$. Therefore, fuzzy extensionality holds in the model $M_n$ over $U_X$, even if $U_X$ is countably infinite, because the comparison is restricted to the finite domain $V_n$, which allows for a clear application of the principle of extensionality.
\end{proof}

\end{lemma}

\begin{lemma} {Verification of \emph{Restricted Stratified Fuzzy} Comprehension } Axiom.

\begin{proof}
Since $V_n$ is finite and $D$ is countable, the set $V_n \times D$ is also countable \footnote{This conclusion follows from standard results in set theory: the Cartesian product of a finite set and a countable set retains countability.}.

Let us consider the formula:
\begin{align*}
\Psi_i(x,d) = (\Phi_i (x,d) \land \forall v \in D (v > d \rightarrow \lnot \Phi_i (x,v))) \land \\ 
\exists v \in D ( v \leq d \land \Phi_i (x,v))) \lor (d = 0 \land \forall v \in D \lnot \Phi_i (x,v))
\end{align*}
This formula is first-order, does not include any self-referential definitions, and is stratified (by hypothesis) because $\Phi_i(x,d)$ is stratified, and logical combinations of stratified formulas maintain stratification.

Let's now apply Axiom~\ref{AxCompDifZF} (\emph{Fuzzy} Comprehension) from \emph{Fuzzy} $\mathit{ZF}$. Given a first-order predicate \(\Psi_i(x,d)\), the existence of a fuzzy set \(A_i \in U_X\) is guaranteed. This means there exists a function \(A_i: V_n \rightarrow D\) such that for any \(x \in V_n\), the following holds: 

\[
\mu_M(x, A_i) = \sup \{d \in D \mid \Psi_i(x, d)\}
\]

Note that the operator \(\sup\) is not used in Axiom~\ref{ax:StrFuzzyComp}; we merely simulate its behavior within the framework of first-order logic.

However, since the definition $\Psi_i(x,d)$ selects the greatest value of $d$, among those for which $\Phi_i(x,d)$ holds, this supremum is both attained and unique by construction. Therefore, we can explicitly define:
$$
\mu_M(x,A_i) = d \Leftrightarrow \Psi_i(x,d)
$$

We conclude that \(A_i \in U_X\), which is defined by comprehension according to a stratified and first-order formula, satisfies precisely Axiom~\ref{ax:StrFuzzyComp} (\emph{Stratified Fuzzy} Comprehension). Furthermore, since \(\Sigma_n\) contains a finite number \(k\) of comprehension formulas, the total set \(S = \{V_X, A_1, \ldots, A_k\} \subseteq U_X\) has a finite cardinality. This set is sufficient to interpret all the formulas of \(\Sigma_n\) within the model \(M_n\).

With this, we complete the lemma.
\end{proof}
\end{lemma}

\begin{lemma} {Verification of \emph{Restricted Fuzzy} Universality Axiom.}
\label{ResFuzzyUniversality}

In the model $M_n$ constructed within \emph{Fuzzy} $\mathit{ZF}$, there exists a set $V_X \in U_X$ such that the following property holds:
\begin{align*}
(\forall x \in V_n, \: \mu_{M} (x, V_X) = 1) \land \mu_{M} (V_{X}, V_{X}) = 1
\end{align*}
\begin{proof}
Given that $V_n = \{x_1, \ldots, x_k\}$ is a finite set of variables, that the set of membership values is given by $D$, and that the set $U_X$ is well defined by Lemma 3.1, each element $A \in U_X$ can be interpreted as a fuzzy set defined over $V_n$.


We are looking for a specific element \( V_X \in U_X \) that acts as a universal set within the restricted domain \( V_n \). This element has a membership function that is constant and equal to 1 for every \( x \in V_n \). Additionally, it must satisfy the condition \( \mu_M(V_X, V_X) = 1 \).

To achieve this:
\begin{Definition}
We define the constant function \( V_X: V_n \rightarrow D \) by setting \( V_X(x) = 1 \) for every \( x \in V_n \).  
\end{Definition}

This function exists because \( 1 \in D \), and thus \( V_X \in U_X \). By construction, it follows from Lemma~\ref{ResFuzzyUniversality}.

It remains to verify that $\mu_{M}(V_{X}, V_{X}) = 1$. Recall that, in the model $M_n$, the membership function $\mu_M$ can also be defined on pairs of elements from $U_X$, provided that the extension of the action of functions $\mu \in U_X$ to other fuzzy sets is appropriately specified.

We adopt the following natural convention, which is compatible with the semantics of the model and with the axioms of \emph{Fuzzy} $\mathit{ZF}$: if $A \in U_X$ is a function $A: V_n \rightarrow D$, and we wish to define its value on a set $B \in U_X$, we do so by extending $A$ to the domain $V_n \cup S$, where $S \subseteq U_X$ is a finite set containing the relevant elements of the model, including $B$. This extension can be justified using the \emph{Fuzzy} Comprehension Axiom of \emph{Fuzzy} $\mathit{ZF}$, since it allows the construction of fuzzy subsets defined by non-bivalent predicates, such as those expressing degrees of membership $\mu(x, A) = d$.

In particular, we define $V_X(V_X) = 1$. This means that $\mu_M (V_X, V_X) = 1$, thus fulfilling the self-membership required by Axiom 3.3. Therefore, the axiom is satisfied\footnote{The construction of $V_X$ does not give rise to Russell's paradox nor does it activate the classical axioms of $\mathit{ZF}$, because, as previously noted, in \emph{Fuzzy} $\mathit{ZF}$, the expression $\mu(x, A) = 1$ is not equivalent to the classical membership relation $x \in A$. Recall that the universal set we have defined is fuzzy.}.
\end{proof}
\end{lemma}

\begin{lemma} {Satisfability of finite subsets in \emph{Fuzzy} $\mathit{ZF}$}
\label{SatisFiniteSubsets}

Every finite subset $\Sigma_n \subseteq \Sigma$ of the axioms of \emph{Fuzzy} $\mathit{NF}$ is satisfiable in the language $\mathcal{L}_2$.
\begin{proof}
Let $\Sigma_n \subseteq \Sigma$ be a finite subset of axioms of \emph{Fuzzy} $\mathit{NF}$. By the construction in this section, for each $\Sigma_n$, there exists a finite model $M_n$ in \emph{Fuzzy} $\mathit{ZF}$. This model satisfies the axioms of Restricted \emph{Fuzzy} $\mathit{NF}$, which are restricted versions of the axioms in $\Sigma_n$ adapted to the finite domain $V_n$.

But a key concern arises regarding the domain restriction: in \emph{Fuzzy} $\mathit{NF}$, the axioms in $\Sigma$ 
involve universal quantifiers over a potentially infinite universe $U$, whereas in $M_n$, the domain is finite. For $M_n$ to model $\Sigma_n$, the universal quantifiers in $\Sigma_n$ must be reinterpreted over this finite domain. This reinterpretation restricts the scope of the axioms, raising the question of whether satisfying theses restricted versions is equivalent to satisfying the original axioms over $U$. 

This concern is addressed by the semantics of FOL: a model satisfies a set of formulas related to its own domain. In $M_n$, the formulas in $\Sigma_n$ are satisfied when their quantifiers range over $V_n \cup S$.
The axioms of Restricted \emph{Fuzzy} $\mathit{NF}$ are designed to align with this interpretation, ensuring that $M_n$ satisfies $\Sigma_n$ under the standard semantics of $\mathcal{L}_2$. While the axioms Restricted \emph{Fuzzy} $\mathit{NF}$ are weaker than those of \emph{Fuzzy} $\mathit{NF}$ due to the finite domain, this weakening is intentional for purpose of our proof.

The legitimacy of this approach hinges on the Compactness Theorem, which only requires that each finite subset $\Sigma_n \subseteq \Sigma$ be  satisfiable in some model. The theorem does not demand that these models satisfy the axioms in their unrestricted form over an infinite domain, nor that the models be identical across all $\Sigma_n$. Here, each $M_n$ provides a consistent interpretation of $\Sigma_n$ within its finite domain, fulfilling this requirement.

Thus, it is valid to claim that every finite subset $\Sigma_n \subseteq \Sigma$ has a model, because $M_n$ satisfies the formulas in $\Sigma_n$ when interpreted over $V_n \cup S$.
 
\end{proof}
\end{lemma}

\begin{theorem}
\label{theorem_1}
Every finite subset $\Sigma_n \subseteq \Sigma$ of axioms of \emph{Fuzzy} $\mathit{NF}$ has a finite model in \emph{Fuzzy} $\mathit{ZF}$. Specifically, for every $\Sigma_n \subseteq \Sigma$, there exists a finite model $(X,U_X)$ in \emph{Fuzzy} $\mathit{ZF}$ that satisfies $\Sigma_n$ with $D = \mathbb{Q} \cap [0,1]$.

\begin{proof}[Proof.]{\textbf{}}
The subset $S=\{V_{X},A_{1},\ldots,A_{k}\}\subseteq U_X$ used in $M_n$ is finite ($\lvert S\rvert =k+1$), although $U_X = \{ \mu : V_n \rightarrow D \}$ are countably infinite ($\left|D\right|=\left|U_{X}\right|= \aleph_0$). This allows for an explicit construction of the model $M_n = (U_{X},V_{n},\mu_{M},D)$ in \emph{Fuzzy} $\mathit{ZF}$, where $V_n$ is a finite, well-founded crisp domain by $\mathit{ZF}$. 
$U_X$ is defined by the Replacement and Power axioms in $D$, which is countable. 

Furthermore, $M_n$ is consistent with \emph{Restricted Fuzzy} $\mathit{NF}$, because $V_X \in S$ satisfies universality with respect to:
$$
V_n (\forall x \in V_{n} (\mu_{M}(x,V_{X})=1)) \land \mu_{M}(V_X,V_X) = 1)
$$
and the $A_i$ satisfy the finitely many instances of stratified comprehension in $\Sigma_n$.

Therefore, the theorem is proved for every finite subset $\Sigma_n \subseteq \Sigma$ of axioms of \emph{Fuzzy} $\mathit{NF}$, so there exists a model $M_n$ in \emph{Fuzzy} $\mathit{ZF}$ with $D=\mathbb{Q} \cap [0,1]$, where $V_n$ is finite. Additionally, \(M_n\) satisfies \(\Sigma_n\) using a finite subset of \(U_X\).
\end{proof}
\end{theorem}

\section{Satisfability of \emph{Fuzzy} $\mathit{NF}$ via Compactness}
\label{modinfinitos}
Our goal in this section is to prove that the entire theory of \emph{Fuzzy} $\mathit{NF}$, denoted by $\Sigma$, has an infinite model by using the Compactness Theorem of first-order logic.
This infinite model, which we will denote as $\mathcal{M}$, will serve as the basis for constructing a crisp submodel satisfying $\mathit{NF}$ in Section \ref{extracModNitido}. The central argument is that, having shown in Theorem \ref{theorem_1} that every finite subset $\Sigma_{n} \subseteq \Sigma$ has a finite model in \emph{Fuzzy} $\mathit{ZF}$, and since $\Sigma$ is a countable set of sentences, the Compactness Theorem guarantees the existence of a model for $\Sigma$.

\begin{Definition}
The Compactness Theorem in first-order logic states that if every finite subset of a countable set of sentences has a model, then the entire set has a model.
\end{Definition}

We apply this result to $\Sigma$ as follows:

\begin{lemma}{Verification of the compactness conditions of $\Sigma$ expressed in $\mathcal{L}_2$.} \label{lemaVerCondComp}
\begin{proof}
Countability of the language $\mathcal{L}_2$, which consists of the following elements:
\begin{itemize}
\item Finite logical symbols: $\forall, \exists, \rightarrow, \land, \lor, \lnot$.
\item Specific symbols: $\mu$ (fuzzy membership function), $=$ (equality), and a constant $d$ for each $d \in D$ ($\left|D\right|= \aleph_0$).
\item Variables, which are countable in number ($\aleph_0$)
\end{itemize}
The entire set of symbols is countable, because the union of a finite set and a countably infinite set ($D$) has cardinality $\aleph_0$.
The countability of $\Sigma$ follows from the fact that it contains only one sentence for extensionality and one for universality, together with one sentence for each stratified formula $\Phi_{i}(x, v)$ in $\mathcal{L}_2$ corresponding to Stratified Comprehension.

Since $\mathcal{L}_2$ is countable, the set of all formulas ($\Sigma$) is countable, and the stratified formulas form a countable subset. Thus: $\left | \Sigma \right | = \aleph_0$, and therefore, verifying the compactness conditions.

The Compactness Theorem applies because: (1) the pseudo-fuzzy logic we employ is an extended many-sorted first-order logic that preserves the fundamental semantics (since the truth values of formulas and the connectives are classical) and the deductive apparatus of first-order logic (FOL); and (2) because the set of membership degrees is countable.
\end{proof}
\end{lemma}

\begin{lemma}{Application of the Compactness Theorem to obtain a model \- $\mathcal{M}$ for $\Sigma$.} 
\label{lem:AppThCompac}
\begin{proof}
Since $\Sigma$ is countable (Lemma \ref{lemaVerCondComp}) and every finite subset $\Sigma_{n} \subseteq \Sigma$ has a finite model $M_n$ in \emph{Fuzzy} $\mathit{ZF}$ (Theorem \ref{theorem_1}). According to the Compactness Theorem, there is a model $\mathcal{M} = (U, \mu_{\mathcal{M}},D)$ that satisfies, simultaneously, all sentences of $\Sigma$. Hence, $U$ is the universe of $\mathcal{M}$, including the definitions of $\mu_{\mathcal{M}}$, and with $D=\mathbb{Q} \cap [0,1]$.
\end{proof}
\end{lemma}

Since $\mathcal{M} \vDash \Sigma$, we verify below that it satisfies the axioms of \emph{Fuzzy} $\mathit{NF}$.

\begin{lemma}{Verification of \emph{Fuzzy Extensionality} Axiom.}

In the infinite model $\mathcal{M}$, constructed via the Compactness Theorem, the axiom of Fuzzy Extensionality of \emph{Fuzzy} $\mathit{NF}$ holds:
$$
\forall A \in U \: \forall B \in U \: (\forall x \in U \, (\mu_\mathcal{M} (x, A) = \mu_{\mathcal{M}}(x, B)) \rightarrow A = B)
$$
\begin{proof}
By construction, the model $\mathcal{M}$ is a first-order model that satisfies the entire theory $\Sigma$ of \emph{Fuzzy} $\mathit{NF}$, as a direct consequence of the Compactness Theorem. In particular, Axiom \ref{fuzzyExtensionality} is satisfied.

The Axiom \ref{fuzzyExtensionality} has a clear semantic interpretation: two sets $A$ and $B$ are equal \emph{iff} their membership functions coincide on every element of the universe $V$.

The validity of this axiom in $\mathcal{M}$ is not an assumption, but rather a direct consequence of the fact that every finite subset $\Sigma_n \subseteq \Sigma$ contains it, and therefore, their finite models $M_n$ already satisfy it. 
\end{proof}
\end{lemma}

\begin{lemma}{Verification of \emph{Stratified Fuzzy} Comprehension Axiom.} \label{lemmaNFCompDifunrestricted}

Let $\mathcal{M}$ be a model such that for every stratified formula $\Phi_i(x,d)$ with $x \in U$ and $d \in D$, there exists a fuzzy set $A_i \in U$ that satisfies Axiom~\ref{ax:StrFuzzyComp} for all $x \in U$.
\begin{proof}
As shown in Lemma~\ref{lem:AppThCompac}, each finite subset $\Sigma_n \subseteq \Sigma$ has a finite model $M_n$ that satisfies \emph{Restricted Fuzzy} $\mathit{NF}$. Therefore, the model satisfies Axiom~\ref{ax:ResFuzzStraCom}

By the Compactness Theorem, we can assert that $\mathcal{M}$ satisfies this property for all formulas in $\Sigma$ without exception.
\end{proof}
\end{lemma}

\begin{lemma}{Verification of \emph{Fuzzy} Universality Axiom.}
\label{lemmNFUniDifunrestricted}

In the model $\mathcal{M}$ of \emph{Fuzzy} $\mathit{NF}$, there exists a set $V_\mathcal{M} \in U$ such that:
$$ \forall x \in U \, (\mu_{\mathcal{M}}(x, V_{\mathcal{M}}) = 1 \land \mu_{\mathcal{M}}(V_{\mathcal{M}},V_{\mathcal{M}}) = 1)
$$
\begin{proof}
As demonstrated in Lemma~\ref{ResFuzzyUniversality}, every finite set of variables $V_n \subseteq U$ admits, within \emph{Fuzzy} $\mathit{ZF}$, a restricted universal set $V_X \in U_X$ such that:
$$
\forall x \in V_n (\mu (x, V_X) = 1 \land \mu (V_X, V_X) = 1)
$$
In other words, Axiom~\ref{ax:resFuzzyUniversality} is satisfiable within each finite model $M_n$ of \emph{Fuzzy} $\mathit{ZF}$ (by application of Lemma~\ref{ResFuzzyUniversality}).

Since the Axiom~\ref{ax:resFuzzyUniversality} is a generalization of its finite version, the Compactness Theorem guarantees the existence of an infinite model $\mathcal{M} \vDash \Sigma$ that satisfies it.
This follows because, by the Compactness Theorem, if every finite subset of the theory $\Sigma$ (which includes the finite instances of Axiom~\ref{ax:resFuzzyUniversality}) is satisfiable, then the entire theory $\Sigma$ is also satisfiable, ensuring the existence of a model-possibly infinite-in which Axiom 3.3 holds.
\end{proof}
\end{lemma}


\begin{theorem}
\label{theorem_2}
The $\Sigma$ theory of \emph{Fuzzy} $\mathit{NF}$ has an infinite model with $D= \mathbb{Q} \cap [0,1]$.

\begin{proof}
By the above lemmas, we began with \emph{Fuzzy} $\mathit{ZF}$ to construct finite models $M_n$ that satisfy every finite subset $\Sigma_n \subseteq \Sigma$ of \emph{Fuzzy} $\mathit{NF}$ (Theorem~\ref{theorem_1}). By applying the Compactness Theorem for first-order logic, we extend this partial satisfiability to an infinite model $\mathcal{M} = (U, \mu_{\mathcal{M}}, D)$ that simultaneously satisfies all sentences in $\Sigma$. This model $\mathcal{M}$ is not necessarily well-founded, nor does it coincide with $\mathit{ZF}$ or \emph{Fuzzy} $\mathit{ZF}$; rather, it is consistent with \emph{Fuzzy} $\mathit{NF}$, as it does not impose such restrictions and includes all instances of stratified comprehension. Furthermore, since both the language $\mathcal{L}_2$ and $\Sigma$ are countable (Lemma~\ref{lemaVerCondComp}) and every $\Sigma_n$ is satisfiable, it follows that $\Sigma$ has an infinite model $\mathcal{M}$ satisfying \emph{Fuzzy} $\mathit{NF}$. Therefore, the theorem is proved.
\end{proof}
\end{theorem}

\section{Extraction of a crisp submodel.}
\label{extracModNitido}
In this section, we construct a crisp submodel $\mathrm{N}$ from the infinite model $\mathcal{M}$ of \emph{Fuzzy} $\mathit{NF}$ ($\Sigma$), such that $\mathrm{N}$ satisfies Extensionality and  Stratified Comprehension axioms of $\mathit{NF}$, in addition to the existence of a universal set $V_n$ with $V_n \in V_n$. A crisp model is one where the fuzzy membership function $\mu$ takes only values of $\{0,1\}$, eliminating \emph{fuzziness}. The proof of relative consistency is now complete; thus, if $\mathit{ZF}$ is consistent, then \emph{Fuzzy} $\mathit{ZF}$ is consistent (by its axiomatic extension); if \emph{Fuzzy} $\mathit{ZF}$ is consistent, then \emph{Fuzzy} $\mathit{NF}$ is consistent (Theorems \ref{theorem_1} and \ref{theorem_2}); and if \emph{Fuzzy} $\mathit{NF}$ is consistent, then $\mathit{NF}$ is also consistent, since $\mathrm{N}$ is a model of $\mathit{NF}$ derived from $\mathcal{M}$ by compactness and this construction.
\begin{Definition} \label{step61}
We define the crisp universe of $\mathrm{N}$ as:
$$
U_{\mathrm{N}} = \{A \in U \mid \forall x \in U \, (\mu_{\mathcal{M}}(x, A) \in \{0,1\})\}
$$
where $U$ is the universe of $\mathcal{M}=(U, \in_{\mathcal{M}},\mu_{\mathcal{M}},D)$. Thus, $U_{\mathrm{N}} \subseteq U$ contains only those elements of $\mathcal{M}$, which fuzzy membership is binary (\emph{Crisp set}).

The structure of $\mathrm{N}$ is:
$$
\mathrm{N}=(U_{\mathrm{N}},\in_{\mathcal{M}})
$$
where the classical membership relation is defined by:
$$
x \in_\mathrm{N} A \Leftrightarrow \mu_{\mathcal{M}}(x,A)=1, \text{ for  } \, x,A \in U_{\mathrm{N}}
$$
\end{Definition}


The fuzzy function $\mu_\mathcal{M}$ turns into a binary relation $\in_\mathcal{M}$ by filtering $\mathcal{M}$, resulting in a classical structure that meets the axioms of \emph{Classic} $\mathit{NF}$, as shown below.

\begin{lemma}{Verification of Extensionality Axiom.}
\begin{proof}
Let $A,B \in U_{\mathrm{N}}$ be such that $\forall x \in \, U_{\mathrm{N}} (x \in_{\mathrm{N}} A \leftrightarrow x \in_{\mathrm{N}} B)$.

By definition, this means:
$$
\forall x \in \, U_{\mathrm{N}} (\mu_{\mathcal{M}}(x, A) = 1 \leftrightarrow (\mu_{\mathcal{M}}(x, B) = 1)
$$
Since $A,B \in U_\mathrm{N}$, we have that $\mu_{\mathcal{M}}(x, A),\mu_{\mathcal{M}}(x, B) \in \{0,1\}$ for all $x \in U$, and the hypothesis implies:
$$
\forall x \in U_{\mathrm{N}} (\mu_{\mathcal{M}}(x, A) = \mu_{\mathcal{M}}(x, B))
$$
In $\mathrm{N}$, which is a classic model with universe $U_\mathrm{N}$, the elements of sets in $U_\mathrm{N} $ determine their unique identity according to the relation $\in_{\mathrm{N}}$. Since $A$ and $B$ overlap in all $x \in U_{\mathrm{N}}$ with respect to $\in_{\mathrm{N}}$, they are equal in $\mathrm{N}$, regardless of their behavior outside $U_{\mathrm{N}}$. 
But it is possible for two sets \( A \) and \( B \) in \( U_{\mathrm{N}} \) to have exactly the same members under the membership relation \( A, B \in U_{\mathrm{N}} \) (i.e., \(\forall x \, \mu(x, A) = \mu(x, B), \, x \in U_{\mathrm{N}}\)), however not be considered equal as elements of the model \(\mathcal{M}\). This situation can arise if the sets differ in their membership values with respect to elements outside \( U_{\mathrm{N}} \). As a result, even though \( A \) and \( B \) contain the same elements within the framework of the model \(\mathrm{N}\), the equality \( A = B \) does not hold. This scenario would violate the Axiom of Extensionality.

This issue can be resolved using the extensional quotient by applying equivalence relations. 
\begin{Definition}
We define an equivalence relation, denoted as $\sim$, on the set $U_\mathrm{N}$ according to the rule: $A \sim B$, if and only if $\mu(x, A) = \mu(x, B)$.
\end{Definition}
In other words, two sets are considered equivalent if they contain the same members in a crisp sense.

This allows us to identify sets that, although distinct in $\mathcal{M}$, are indistinguishable from the perspective of the crisp model.

\begin{Definition}
The universe of the model \(\mathrm{N}\) is defined as the set of equivalence classes \(U_{\mathrm{N}}/{\sim}\), which provides the domain for the new extensional model.
\end{Definition}
\begin{Definition}
The membership relation for these equivalence classes is defined as follows:
\(x \in_{\mathrm{N}} [A]\) if and only if \(\mu(x, A) = 1\), 
where \(A \in_{\mathrm{N}} U_{\mathrm{N}}\) is any representative of the class \([A]\).
\end{Definition}
Once this quotient model $\mathrm{N} = (U_\mathrm{N}/{\sim}, \in_\mathrm{N})$ has been constructed, the Axiom of Extensionality is satisfied automatically and trivially, since if two equivalence classes $[A]$, $[B]$ have the same elements in the sense of $\in_\mathrm{N}$, then their representatives $A$ and $B$ satisfy $\mu(x, A) = \mu(x, B)$ for all $x \in U_\mathrm{N}$, so $A \sim B$ and therefore $[A] = [B]$.
\end{proof}
\end{lemma}

\begin{Definition} \label{def:auxiliary}
Consider $\Phi_i (x,y)$ in $\mathcal{M}$ as an auxiliary formula. Let $\Phi(x)$ be a stratified formula in the language of $\mathit{NF}$. In $\mathcal{M}$, it operates jointly a membership function $\mu_{\mathcal{M}} : U \times U \rightarrow D$ (where $D = \mathbb{Q} \cap [0,1]$), its definition is:
$$
\Phi_i (x,y) = (y=1 \land \Phi(x))
$$
where $y$ is a \emph{type 0} variable, with values in $D$.
$\Phi (x)$ maintains its stratification because the types assigned to the variables in $\Phi (x)$ (according to the definition of stratification in $\mathit{NF}$) are not altered by adding the condition $y=1$. 
\end{Definition}

\begin{lemma} {Verification of Stratified Comprehension Axiom.}

Let $\Psi(x)$ be a stratified formula in the classical language with $\in$. Then there exists a set $[A] \in U_\mathrm{N}/\sim$ such that in the crisp model $\mathrm{N}$, the following holds:
$$
x \in_\mathrm{N} [A] \leftrightarrow \Psi(x)
$$

\begin{proof}
The fuzzy model $\mathcal{M}$ satisfies the Stratified Fuzzy Comprehension Axiom (Lemma \ref{lemmaNFCompDifunrestricted}). Given that $\Psi(x)$ is a classical stratified formula, we define an auxiliary formula to adapt it to Lemma~\ref{lemmaNFCompDifunrestricted}:
$$
\Phi (x, d) = (d = 1 \land \Psi (x))
$$
Furthermore, this $\Phi$ is stratified in $\mathcal{M}$.

The potential cases are as follows:
\begin{itemize}
     \item If $\Psi(x)$ is true, then $\Phi(x, 1)$ is true, and since there is no $v > 1$ in $D$, the left-hand side of the biconditional in Axiom \ref{ax:ResFuzzStraCom} for $d = 1$ is satisfied. Therefore, $\mu(x, A) = 1$.
     \item If $\Psi(x)$ is false, then $\Phi(x, d)$ is false for every $d \in D$, and by the second disjunct of the single disjunction of Axiom \ref{ax:StrFuzzyComp}, it follows that $\mu(x, A) = 0$.
\end{itemize}
Therefore:
\[
\mu_\mathcal{M}(x,A) = \left\{
\begin{array}{l}
1 \: \: \text{   if   } \: \: \Psi(x), \\
0 \: \: \text{   if   } \: \: \lnot \Psi(x).
\end{array}
\Rightarrow (x \in_\mathrm{N} [A] \leftrightarrow \mu (x,A)=1 \leftrightarrow \Psi (x))
\right.
\]

Since $\mu(x, A) \in \{0, 1\}$, it follows that $A \in U_\mathrm{N}$, and therefore $[A] \in U_\mathrm{N}/\sim$. Thus, the model $\mathrm{N}$ satisfies the Stratified Comprehension Axiom.

\end{proof}
\end{lemma}

\begin{lemma}{Verification of Universality Axiom.}

Let $\mathcal{M}$ be an \emph{Fuzzy} $\mathit{NF}$ model, and let $\mathrm{N}$ be a classical submodel defined as the quotient of the crisp subuniverse.
$$
U_\mathrm{N} = \{A \in U \mid \forall x \in U (\mu_\mathcal{M}(x,A) \in \{0,1\}) \},
$$
by the equivalence relation:
$$
A \sim B \Longleftrightarrow \forall x \in U_\mathrm{N} (\mu_\mathcal{M} (x,A) = \mu_\mathcal{M}(x,B))
$$
The universe of $\mathrm{N}$ is the quotient set $U_\mathrm{N}/\sim$, and membership in $\mathrm{N}$ is defined as follows:
$$
[x] \in_\mathrm{N} [A] \Longleftrightarrow \mu_\mathcal{M} (x, A) = 1
$$
Then, there exists a universal set $[V] \in U_\mathrm{N}/\sim$ such that:
$$
\forall [x] \in U_\mathrm{N}/\sim ([x] \in_\mathrm{N} [V]) \text{    and   } [V] \in_\mathrm{N} [V]
$$

\begin{proof}
First, we prove:
\begin{itemize}
\item Existence of $V$ in $\mathcal{M}$.
By Lemma~\ref{lemmNFUniDifunrestricted}, since $\mathcal{M} \vDash \Sigma$, there exists a set $V \in U$ such that:
$$
\forall x \in U (\mu_\mathcal{M} (x, V) = 1) \text {   and   } \mu_\mathcal{M}(V,V)=1
$$

\item Verification that $V \in U_\mathrm{N}$.
Since $\forall x \in U$, it holds that:
$$
\mu_\mathcal{M} (x,V) = 1 \in \{0,1\}
$$
It follows that $V \in U_\mathrm{N}$.

\item $V$ is an equivalence class $[V] \in U_\mathrm{N}/\sim$.
Let $[V]$ denote the equivalence class of $V$ in the quotient set. Since membership is defined by the relation:
$$
[x] \in_\mathrm{N} [V] \Longleftrightarrow \mu_\mathcal{M} (x, V) = 1
$$
and since by hypothesis $\mu_\mathcal{M}(x, V) = 1$ for all $x \in U$, and in particular for all $x \in U_\mathrm{N}$, it follows that:
$$
\forall [x] \in U_\mathrm{N}/\sim([x] \in_\mathrm{N} [V])
$$
Furthermore, since $\mu_\mathcal{M}(V, V) = 1$, it is guaranteed that $[V] \in_\mathrm{N} [V]$. Thus, it has been shown that the equivalence class $[V]$ serves as the universal set in the model $\mathrm{N}$, thereby confirming the Universality Axiom of $\mathit{NF}$.
\end{itemize}
\end{proof}
\end{lemma}




\begin{lemma}{Representability of Stratified Formulas by Crisp Sets.}

Let $\Psi(x)$ be a classical stratified formula with parameters in the crisp submodel $\mathrm{N}$. Then there exists a set $[A] \in \mathrm{N}$ such that for all $[x] \in \mathrm{N}$, 
\[
[x] \in_\mathrm{N} [A] \Leftrightarrow \mathrm{N} \vDash \Psi(x).
\]
\begin{proof}
Since $\Psi(x)$ is stratified and $\mathrm{N}$ is a quotient of a model $\mathcal{M} \vDash \Sigma$, there exists a set $A \in U$ such that
\[
\mu(x,A) = 1 \Leftrightarrow \mathcal{M} \vDash \Psi(x)
\]
(by means of Axiom \ref{ax:StrFuzzyComp}). But by definition of $\in_\mathrm{N}$, this implies that
\[
[x] \in_\mathrm{N} [A] \Leftrightarrow \mathrm{N} \vDash \Psi(x).
\]
Therefore, $[A]$ represents $\Psi(x)$ in the crisp submodel.

With this lemma, we guarantee that the crisp submodel $\mathrm{N}$ has the same expressive power as $\mathit{NF}$, being sufficiently strong to interpret all stratified formulas. Thus, we ensure that every (classical) stratified formula is representable by a crisp set—that is, for every property expressible by a stratified formula, there exists a set whose elements are those that satisfy the formula.
\end{proof}
\end{lemma}

\begin{theorem}
\label{theorem_3}
There exists a crisp model $\mathrm{N}$ constructed from $\mathcal{M}$ that satisfies $\mathit{NF}$.

\begin{proof}
By the previous lemmas, we have established that the crisp submodel \(\mathrm{N}\), constructed from \(\mathcal{M}\), satisfies the Extensionality and Stratified fundamental Axioms of \(\mathit{NF}\). Additionally, it exhibits the derived property of universality: \(\exists V \in U_\mathrm{N} (\forall x \in U_\mathrm{N} (x \in_\mathrm{N} V) \land V \in_\mathrm{N} V)\). Therefore, we can conclude that \(\mathrm{N}\) is a model of $\mathit{NF}$.

Given a model $\mathcal{M}$ of \emph{Fuzzy} $\mathit{NF}$ ($\Sigma$), obtained via the Compactness Theorem in Theorem~\ref{theorem_2}, there exists a crisp submodel $\mathrm{N}$ that satisfies $\mathit{NF}$. Consequently, if $\mathit{ZF}$ is consistent, then \emph{Fuzzy} $\mathit{ZF}$ is also consistent, as it is an axiomatic extension. If \emph{Fuzzy} $\mathit{ZF}$ is consistent, then so is \emph{Fuzzy} $\mathit{NF}$ (Theorems~\ref{theorem_1} and~\ref{theorem_2}). Finally, if \emph{Fuzzy} $\mathit{NF}$ is consistent, then $\mathit{NF}$ is consistent as well, since the existence of $\mathrm{N}$ follows from $\mathcal{M}$.

After verifying that $\mathrm{N} $ meets the extensionality and the stratified comprehension scheme, it follows that $\mathrm{N}$ also supports all standard constructions of $\mathit{NF}$ derived from stratified formulas, such as \emph{singletons} ($\{a\} = \{x \mid x=a\}$), Cartesian products ($A \times B$), inverse relations, projections, relative products, and tuple insertions, among others. It confirms that $\mathrm{N}$ is a complete model of  $\mathit{NF}$, completing Theorem~\ref{theorem_3} definitively.
\end{proof}

\end{theorem}

\section{Conclusions}
\label{conclusions}
In set theory, the consistency of $\mathit{NF}$ places this theory as a feasible alternative to $\mathit{ZFC}$, offering an intuitive perspective for handling large-scale collections, such as those encountered in category theory. Using \emph{Fuzzy Forcing}, we can systematically investigate deeper properties of $\mathit{NF}$, including its cardinal arithmetic and ordinal behavior. This approach opens up new research avenues in the foundations of mathematics.

From a philosophical perspective, the consistency of $\mathit{NF}$ supports foundational pluralism, which challenges the dominance of $\mathit{ZFC}$. This point of view encourages the selection of foundational frameworks according to the specific requirements of each problem, much like the impact of non-Euclidean geometries, and is consistent with the movement toward logical pluralism~\cite{Beall2000-BEALP}.

Beyond $\mathit{NF}$, Fuzzy Forcing emerges as a powerful method to model and study mathematical structures in a broad way, due to its capacity and versatility in establishing connections between disparate structures. This suggests significant potential beyond set theory, with promising applications in category theory, topology, algebra, and discrete mathematics. For instance, in graph theory, it could be employed to investigate the existence of specific graphs with particular properties (e.g., fuzzy [or non-fuzzy] graphs exhibiting extremal or non-standard logical and axiomatic properties). Furthermore, within discrete mathematics and set theory, it facilitates the analysis of \emph{unconventional} theories, such as $\mathit{ZF}$ augmented with Aczel’s Anti-Foundation Axiom and other rare set-theoretic variants.

We are currently exploring these applications in several ongoing studies, with preliminary results to be presented in future publications. We hope that this method will inspire novel research, foster collaborative efforts to explore its potential across diverse areas of pure mathematics, and contribute to generating scientific knowledge applicable to engineering and other disciplines.

\bibliographystyle{plainurl}
\bibliography{referencias}

\end{document}